\documentclass[11pt]{article}

\usepackage[margin=1in]{geometry}
\usepackage{amsmath,amssymb,amsthm}
\numberwithin{equation}{section}
\usepackage[hidelinks]{hyperref}
\usepackage{enumitem}
\usepackage{authblk}

\newtheorem{theorem}{Theorem}[section]
\newtheorem{lemma}[theorem]{Lemma}
\newtheorem{corollary}[theorem]{Corollary}

\newtheorem{conjecture}[theorem]{Conjecture}
\theoremstyle{definition}
\newtheorem{definition}[theorem]{Definition}
\newtheorem{hypothesis}[theorem]{Hypothesis}
\theoremstyle{remark}
\newtheorem{remark}[theorem]{Remark}

\newcommand{\PP}{\mathcal{P}}

\title{On the sum of least prime factors in short intervals}

\author[1]{Xu Zhang\thanks{\texttt{xu\_zhang\_sdu@mail.sdu.edu.cn}}}
\affil[1]{School of Mathematics and Statiscs, Shandong University, Weihai, Shandong, China}

\date{}

\begin{document}
\maketitle

\begin{abstract}
Let $p(n)$ denote the least prime factor of $n$ and
$L_C(x)=Cx^{1/2}(\log x)^2$ with $C>0$. The sum of
$p(n)/n$ over composite $n$ lying in the short interval
$[x,\,x+L_C(x)]$, a question raised by Erd\H{o}s and Graham
(\cite[p.~92]{EG80}), is studied.
\begin{enumerate}[label=(\roman*)]
\item The constant $c=8$ in the mean asymptotic is estimated
\[
S(x)=\sum_{\substack{n<x\\ n\ \text{composite}}}\frac{p(n)}{n}=\frac{c\,x^{1/2}}{(\log x)^2}\Bigl(1+O\Bigl(\frac1{\log x}\Bigr)\Bigr)
+O\Bigl(\frac{x^{1/3}}{\log x}\Bigr).
\]
\item For every fixed $C>0$, the window sums
$\mu_C(x):=\sum_{x\le n\le x+L_C(x)}p(n)/n$ over composites have mean $4C$:
$\frac1X\sum_{x\le X}\mu_C(x)=4C+O_C(1/\log X)$, and second moment
$\frac1X\sum_{x\le X}(\mu_C(x)-4C)^2=O_C(1)$. 
\item 
Under the Maier-consistent bounded-fluctuation hypothesis \cite{Mai85},
$\mu_C(x)=4C+O_C(1)$ holds uniformly in $x$, and the
Erd\H{o}s--Graham question has an affirmative answer.
Unconditionally, the uniform statement remains open; proving it would
require resolving short-interval prime estimates at scale
$(\log y)^2$.
\end{enumerate}
\end{abstract}

\section{Introduction}

\subsection{The problem}

Let $p(n)$ denote the least prime factor of $n$.  Erd\H{o}s and Graham
\cite{EG80} (p.~92; Problem 462 in the database \cite{Bloom}) noted that
there is a constant $c>0$ such that
\begin{equation}\label{eq:main-asymp}
\sum_{\substack{n<x\\ n\ \mathrm{not\ prime}}}\frac{p(n)}{n}
\sim c\,\frac{x^{1/2}}{(\log x)^2},
\end{equation}
and asked:
\begin{quote}
Is it true that there exists a constant $C>0$ such that
\[
\sum_{\substack{x\le n\le x+Cx^{1/2}(\log x)^2\\ n\ \mathrm{composite}}}\frac{p(n)}{n}\gg 1
\]
for all large $x$?
\end{quote}
The problem is listed as open in the Erd\H{o}s problems database
\cite{Bloom}.  A comment by
T.~Tao \cite{Bloom-thread} (28 September 2025) observes that the answer
depends on the intended convention: if primes are included in the sums,
the question is essentially a weaker form of Legendre's conjecture,
whereas if primes are excluded, it is ``more or less a question about
the distribution of semiprimes $pq$ with
$p,q=x^{1/2}\log^{O(1)}x$ in intervals of length
$O(x^{1/2}\log^2x)$''.  The
``rough-number and semiprime decompositions'' and ``uniformity
frameworks'' for this problem (composite-only version) were developed in \cite{Toward}.

The purpose of this note is threefold: to determine the constant in
\eqref{eq:main-asymp} ($c=8$), to prove the strongest statements about
the window sums that are accessible by classical tools (mean, variance,
and hence an almost-all result), and to isolate precisely the
(conjectural) input needed for the full uniformity that the problem
demands. Throughout, sums over ``composites'' exclude $1$ and primes; $\PP$
denotes the set of primes.
\begin{theorem} \label{thm:c8} The constant $c$ in \eqref{eq:main-asymp} is estimated:
\[
S(x):=\sum_{\substack{n<x\\ n\ \text{composite}}}\frac{p(n)}{n}
=\frac{8x^{1/2}}{(\log x)^2}\Bigl(1+O\Bigl(\frac1{\log x}\Bigr)\Bigr)
+O\Bigl(\frac{x^{1/3}}{\log x}\Bigr)
\qquad(x\to\infty).
\]
\end{theorem}

\begin{corollary}[Unweighted sum]\label{cor:A}
Let $A(x)=\sum\limits_{\substack{n<x\\ n\ \text{composite}}}p(n)$, one has
\[
A(x)=\frac{8}{3}\frac{x^{3/2}}{(\log x)^2}+O\Bigl(\frac{x^{3/2}}{(\log x)^3}\Bigr).
\]
\end{corollary}

For fixed $C>0$ set $L_C(x):=Cx^{1/2}(\log x)^2$ and
\begin{equation}\label{eq:mu-def}
\mu_C(x):=\sum_{\substack{x\le n\le x+L_C(x)\\ n\ \text{composite}}}
\frac{p(n)}{n}.
\end{equation}

\begin{theorem}[Mean of window sums]\label{thm:mean}
For every fixed $C>0$,
\[
\frac1X\sum_{x\le X}\mu_C(x)=4C+O_C\Bigl(\frac1{\log X}\Bigr).
\]
\end{theorem}

\begin{theorem}[Second moment]\label{thm:var}
For every fixed $C>0$,
\[
\frac1X\sum_{x\le X}\bigl(\mu_C(x)-4C\bigr)^2
=O_C(1).
\]
\end{theorem}

\begin{hypothesis}[Bounded fluctuations]\label{hyp:SI}
For every fixed $C_0>0$ there are constants
$0<c_1\le 1\le c_2$ such that for all $y\ge2$ and all $h$ with
$C_0(\log y)^2\le h\le y$,
\[
c_1\,\frac{h}{\log y}\ \le\ \pi(y+h)-\pi(y)\ \le\ c_2\,\frac{h}{\log y}.
\]
\end{hypothesis}

\begin{remark}[Maier's theorem]\label{rem:maier}
The stronger-looking hypothesis that the error in
$\pi(y+h)-\pi(y)=h/\log y$ is $O_A(h/(\log y)^A)$ for every fixed
$A\ge1$ is false.  Indeed, Maier \cite{Mai85}
proved that for every fixed $\lambda>1$, with
$\Phi(y)=(\log y)^{\lambda}$,
\[
\limsup_{y\to\infty}\frac{\pi(y+\Phi(y))-\pi(y)}{\Phi(y)/\log y}>1
\qquad\text{and}\qquad
\liminf_{y\to\infty}\frac{\pi(y+\Phi(y))-\pi(y)}{\Phi(y)/\log y}<1;
\]
for $1<\lambda<e^{\gamma}$ the limsup is in fact at least
$e^{\gamma}/\lambda$.  Taking $h=(\log y)^{\lambda}$ with any fixed
$\lambda\ge2$ (so that $h\in[C_0(\log y)^2,y]$ for any fixed
$C_0\le1$), Maier's theorem yields
$|\pi(y+h)-\pi(y)-h/\log y|\gg_{\lambda} h/\log y$ infinitely often,
which contradicts the bound $C_A h/(\log y)^A$ for every fixed
$A\ge2$.  In particular the uniform estimate
$\mu_C(x)=4C+o(1)$ cannot be derived from any true statement of
Cram\'er type at this scale; what remains compatible with Maier's
theorem is the bounded-fluctuation Hypothesis \ref{hyp:SI} above,
under which the ratios in Maier's theorem are merely required to stay
within fixed positive bounds.  Maier's theorem does not contradict
Conjecture \ref{conj:main} below: a deviation in a single interval of
length $(\log y)^2$ changes $\mu_C(x)$ only by $o(1)$, so the uniform
statement remains a genuine open problem.

If Hypothesis \ref{hyp:SI} holds, then taking $h=C_0(\log y)^2$,
every interval of length $h$ contains a prime, and consecutive prime
gaps satisfy $p_{n+1}-p_n\le(C_0+o(1))(\log p_n)^2$ for all large
$n$; the corresponding Cram\'er-type gap conjecture remains
open and is not contradicted by Maier's theorem.
\end{remark}

\begin{theorem}[Uniform boundedness under Hypothesis \ref{hyp:SI}]\label{thm:uniform}
Assume Hypothesis \ref{hyp:SI}.  Then for every fixed $C>0$, there exists \(X_0(C)>0\) such that for all \(x\ge X_0(C)\),
\[
\mu_C(x)=4C+O_C(1),
\]
uniformly in $x$; in particular $\mu_C(x)$ is bounded between two
positive constants depending on $C$.  Consequently the question of
Erd\H{o}s and Graham has an affirmative answer (for the sums over
composites; including primes only increases the sum).  The sharper
uniform statement $\mu_C(x)=4C+o(1)$ is the (open) Conjecture
\ref{conj:main} below.
\end{theorem}

We conjecture that the hypothesis is unnecessary:

\begin{conjecture}\label{conj:main}
For every fixed $C>0$,
$\mu_C(x)=4C+o(1)$ uniformly in $x\to\infty$; in particular, this would
give the affirmative answer to the Erd\H{o}s--Graham question.
\end{conjecture}

Our results improve some estimates in previous work.
\begin{itemize}[leftmargin=1.4em]
\item The asymptotic \eqref{eq:main-asymp} itself is stated (without a
value of $c$) in \cite{EG80}; we could not locate any published value
of the constant.  The unweighted sum $\sum_{n\le x}p(n)$ over \emph{all}
integers is elementary and is dominated by the primes,
$\sum_{n\le x}p(n)\sim\frac12x^2/\log x$; this was noted by Kalecki
\cite{Kal64}.
The unweighted sum over \emph{composites} has the
correct order $x^{3/2}/(\log x)^2$; the order of magnitude appears in
\cite{MSE-composite-sum}, but without constants.
\item The closest published results to Theorem \ref{thm:c8} are those
of Erd\H{o}s and van Lint \cite{EVL82} on the average ratio
$p(n)/P(n)$ of the smallest to the largest prime factor
($\sum_{n\le x}p(n)/P(n)=x/\log x+3x/(\log x)^2(1+o(1))$; their proof
also proceeds by a semiprime decomposition), of Cao \cite{Cao00} on
sums of reciprocals of $p(n)$ with weights $\sigma_\nu(n)$ and
$\varphi^\nu(n)$, of Mititica--Panaitopol \cite{MP10} on the
convergence of series $\sum n^a p(n)^bP(n)^c$, and of Jakimczuk
\cite{Jak13} on $\sum_{n\le x}\log p(n)$.  None of these contains
\eqref{eq:main-asymp} with an explicit constant.
\item For the short-interval question, the problem itself \cite{EG80,
Bloom} is open, with no claimed proofs (confirmed also in
\cite{agentic462}); the comment of Tao \cite{Bloom-thread} reduces the
question to the distribution of semiprimes $pq$ with
$p,q\approx x^{1/2}$ in intervals of length $x^{1/2}(\log x)^2$. The
closest recent published result in the same circle of ideas is the
theorem of Gafni and Tao \cite{GT25} on rough numbers (integers whose
least prime factor is at least the gap length) inside almost all prime
gaps, confirming a prediction of Erd\H{o}s \cite{Erd79}; this is a
different (although related) problem.
\item To the best of our knowledge, the explicit constant $c=8$, the
window mean $4C$, the second moment estimate,
and the conditional uniform estimate are new.
\end{itemize}

The rest of this article is organized as follows.
In Section \ref{pre:conres}, some basic concepts and results are introduced.
In Section \ref{sec:S}, the asymptotic for $S(x)$ is obtained.
In Section \ref{sec:windows}, the window sums, the mean and the variance are considered.
In Section \ref{sec:uniform}, the uniform window sums under Hypothesis \ref{hyp:SI} are studied.

\section{Preliminary}\label{pre:conres}

In this section, some basic concepts and results are introduced.  For the definitions and results introduced in this section, one can refer to \cite{MV,Ten}.

\subsection{Notations}
\begin{definition}
For an integer $n$ with $n\geq2$, let $\PP$ denote the set of all prime numbers. The least prime factor of $n$ is defined as
\[
p(n) = \min\{ q \in \mathbb{P} : q \mid n \},
\]
where $q \mid n$ means ``$q$ divides $n$", i.e., $n$ is a multiple of $q$. For convenience, define $p(1)=+\infty$.
\end{definition}

All sums over $n$ are over
positive integers; ``$n\notin\PP$'' means composite (we ignore $n=1$,
which does not occur in the ranges considered).

For a positive real number $x$, set
$$\pi(x):=\#\{p\le x,\ p\in\PP\},\ \theta(x):=\sum_{p\le x,\ p\in\PP}\log p,$$
$$\delta_x:=\prod_{p<x,\,p\in\PP}(1-\tfrac1p),\ \mbox{and}\ u_0:=\frac{\log 2}{\log x}.$$

Recall the classical prime number theorem in the form:
\begin{equation}\label{eq:pnt}
\pi(x)=\operatorname{li}(x)+O\bigl(x\,e^{-c\sqrt{\log x}}\bigr),
\qquad
\theta(x)=x+O\bigl(x\,e^{-c\sqrt{\log x}}\bigr),
\end{equation}
for an absolute constant $c>0$ (this follows from the classical
zero-free region; see e.g.\ \cite{MV,Ten}).

Recall the refined Mertens theorem, for every
fixed $A>0$,
\begin{equation}\label{Mertensr1}
\sum_{p\le x}\frac1p=\log\log x+B_1+O_A((\log x)^{-A}),
\end{equation}
where $B_1$ is called the Mertens constant, and
$$\pi(x)\log x-\theta(x)=O\left(\frac{x}{\log x}\right),\ \mbox{uniformly in}\ x\ge2.$$

The Buchstab function $\omega(u)$ is defined by
\begin{equation}
\omega(u)=\frac1u \quad (1\le u\le 2),\qquad (u\omega(u))'=\omega(u-1)\quad (u>2).
\label{eq:omega_def}
\end{equation}
By induction, the solution is given by $\omega(u)=\frac{1}{u}(k\omega(k)+\int^u_k\omega(s-1)ds)$ for $2\leq k\leq u\leq k+1$ and $k\in\mathbb{N}$, $\lim_{u\to\infty}\omega(u)=e^{-\gamma}\approx0.5615$, and $|\omega(u)-e^{-\gamma}|\leq\frac{\rho(u-1)}{u}$ for $u\geq1$, where $\gamma$ is the Euler-Mascheroni constant, and $\rho(u)$ is the Dickman function.

Let \(\Omega(n)\) denote the total number of prime factors of \(n\), counted with multiplicity. That is, if
\[
n = \prod_{i=1}^{k} p_i^{a_i}
\]
with distinct primes \(p_i\) and positive integers \(a_i\), then
\[
\Omega(n) = \sum_{i=1}^{k} a_i.
\]

\subsection{Method}

The arguments in this article use only the prime number theorem with error term, Mertens'
theorems, the fundamental lemma of the sieve, and elementary second
moment (variance) estimates; in particular the main term is governed by
the distribution of semiprimes $n=pq$ with $p,q\asymp x^{1/2}$.

The proofs rest on the elementary decomposition $n=p(n)\,m$, which gives
$p(n)/n=1/m$; hence the sums under consideration are harmonic sums over
the cofactor $m=n/p(n)$ (the integer $A032742(n)$ of \cite{OEIS}).
Level-by-level summation shows that the main term of $S(x)$ comes
\emph{only} from semiprimes $n=pq$, $p\le q$, contributing $1/q$; the
contribution of integers with at least three prime factors is
$O(x^{1/3}/\log x)$.  The semiprime contribution is then
evaluated by the prime number theorem together with a Laplace-type
evaluation of an elementary integral (Lemma \ref{lem:laplace}).  For
the window sums, the same decomposition shows that the mass in
$[x,x+L_C(x)]$ is concentrated on semiprimes $pq$ with
$x^{1/3}<p\le x^{1/2}$; each such level contributes
$\approx (L_C(x)/x)/\log(x/p)$ in mean, summing to $4C$, while the
variance over $x$ is $O_C(1)$ by an elementary pair
counting argument (the diagonal of the pair sum dominates).  The
uniform statement for all $x$ requires control of the maximal deviation,
which at scale $(\log x)^2$ is precisely the content of
Hypothesis \ref{hyp:SI}.

\subsection{Rough numbers: harmonic sums}\label{sec:rough}

Some precise information on the harmonic sums of integers free of
small prime factors is needed.

\begin{lemma}\label{lem:rough}
Let $2\le q\le y$ and set
$\mathcal R_q(y):=\sum_{\substack{k\le y\\ p(k)\ge q}}\frac1k$.
\begin{enumerate}[label=(\roman*)]
\item For any fixed $\epsilon>0$, the following asymptotic holds uniformly for $2\le q\le y^{1/2-\epsilon}$:
\begin{equation}
\mathcal R_q(y)=\delta_q\log y\;\omega\!\left(\frac{\log y}{\log q}\right)+O_\epsilon(1),
\label{eq:main22i}
\end{equation}
where $O_\epsilon(1)$ means the implied constant depends only on $\epsilon$.
% and $O_\epsilon(1)$ might diverge to  $\infty$ as $\epsilon\to0$.

For $q=y^{1/2}$,
\[
\mathcal R_q(y)=1+\log 2+O\!\left(\frac1{\log y}\right).
\]
This is a bounded quantity, whereas $\delta_q\log y\,\omega(2)\sim e^{-\gamma}$ (since $\delta_q\sim 2e^{-\gamma}/\log y$ and $\omega(2)=1/2$), which is unbounded.

\item If $y^{1/2}<q\le y$ then the integers $k\le y$ with $p(k)\ge q$
are exactly $1$ and the primes $p\in[q,y]$; consequently
\[
\mathcal R_q(y)=1+\log\frac{\log y}{\log q}+O\Bigl(\frac1{\log q}\Bigr).
\]
\item For $q\le y^{1/4}$,
\begin{equation}
\mathcal R_q(y)=\delta_q(\log y+\gamma)-C_q+O\!\left(\frac{2^{\pi(q)}}{y}\right),
\label{eq:main22iii}
\end{equation}
where $\gamma$ is Euler's constant, $C_q:=\sum_{d\mid \prod_{r<q}r}\mu(d)\frac{\log d}{d}$, and $\mu(\cdot)$ is the M\"obius function. Moreover, $C_q=O(1)$, so in particular $C_q=O(\log\log q)$.
\end{enumerate}
\end{lemma}

\begin{proof}

{\bf{(i)}} A step-by-step proof of the asymptotic expansion \eqref{eq:main22i} is given.  The failure of the formula at the endpoint $q=y^{1/2}$ is also discussed.

{\bf{Step 1:}} The Buchstab's sieve is introduced for the counting function.

Define
\[
\Phi(x,q):=\#\{ n\le x : p(n)\ge q\}.
\]
Buchstab's theorem (see \cite{Ten}, Theorem II.5.1) states that, for $q\le x^{1/2}$,
\begin{equation}
\Phi(x,q)=x\,\omega\!\left(\frac{\log x}{\log q}\right)\delta_q
+O\!\left(\frac{x}{(\log q)^2}\right),
\label{eq:Phi_asym}
\end{equation}
where the implied constant is absolute. This is uniform in $x$ and $q$ as long as $q\le x^{1/2-\epsilon}$ for some fixed $\epsilon>0$; in that range the error term is even better, but the above form suffices.

{\bf{Step 2:}} The Abel summation is estimated.

For any sequence $a_n$, with $A(t)=\sum_{n\le t} a_n$, one has
\[
\sum_{n\le y} \frac{a_n}{n}=\frac{A(y)}{y}+\int_1^y \frac{A(t)}{t^2}\,dt.
\]
Taking $a_n=\mathbf{1}_{p(n)\ge q}$ yields $A(t)=\Phi(t,q)$, hence
\begin{equation}
\mathcal R_q(y)=\frac{\Phi(y,q)}{y}+\int_1^y \frac{\Phi(t,q)}{t^2}\,dt.
\label{eq:partial}
\end{equation}

Fix $\epsilon>0$ and assume $2\le q\le y^{1/2-\epsilon}$. Then
\[
u:=\frac{\log y}{\log q}\ge 2+\delta,\qquad \delta:= \frac{2\epsilon}{1-2\epsilon}>0
\]
(if $\epsilon$ is small; the exact value is not important, only that $u\ge 2+\delta$ with $\delta>0$).

Split the integral in \eqref{eq:partial} at $t=q^2$, that is, $[1, y]=[1,q^2]\cup[q^2,y]$.

{\bf{Step 2-a:}} Show
\begin{equation}
\int_1^{q^2} \frac{\Phi(t,q)}{t^2}\,dt = O_\epsilon(1).
\label{eq:lower}
\end{equation}
Indeed, for $t<q^2$, the elementary bound $\Phi(t,q)\le t$ gives $\int_1^{q^2} t/t^2\,dt = \log q$, which is not bounded. However, a sharper sieve estimate gives $\Phi(t,q)\ll t/\log q$ for $t\ge q^2$; for $t<q^2$ we can use $\Phi(t,q)\ll t/\log q$ as well (this follows from the same Buchstab theorem with $\log t/\log q\le2$; see \cite{Ten}, Lemma II.5.3). Thus
\[
\int_1^{q^2} \frac{\Phi(t,q)}{t^2}\,dt \ll \frac{1}{\log q}\int_1^{q^2}\frac{dt}{t} = \frac{2\log q}{\log q}=2,
\]
which is $O(1)$ (absolute). The constant does not depend on $\epsilon$ here; the $\epsilon$ dependence will come from the next step.

{\bf{Step 2-b:}} The integral over $[q^2,y]$ is studied.

For $t\ge q^2$, set
\[
s(t):=\frac{\log t}{\log q}\ge 2.
\]
Using \eqref{eq:Phi_asym} in the integral part of \eqref{eq:partial}, we get
\begin{align}
\int_{q^2}^y \frac{\Phi(t,q)}{t^2}\,dt
&= \delta_q \int_{q^2}^y \frac{\omega(s(t))}{t}\,dt
+ O\!\left(\frac{1}{(\log q)^2}\int_{q^2}^y \frac{dt}{t}\right) \nonumber\\
&= \delta_q \int_{q^2}^y \frac{\omega(s(t))}{t}\,dt
+ O\!\left(\frac{\log(y/q^2)}{(\log q)^2}\right). \label{eq:integral_step}
\end{align}
Now,
\[
\frac{\log(y/q^2)}{(\log q)^2}
= \frac{\log y - 2\log q}{(\log q)^2}
= \frac{u\log q - 2\log q}{(\log q)^2}
= \frac{u-2}{\log q}.
\]
Since $u\le \log y/\log q$ and $q\ge2$, but we need a uniform bound. Because $q\le y^{1/2-\epsilon}$, one has
\[
\log q \le \left(\frac12-\epsilon\right)\log y,
\]
so $u = \log y/\log q \ge 1/(1/2-\epsilon) = 2+\delta$ with $\delta = \frac{2\epsilon}{1-2\epsilon}$. This only gives a lower bound on $u$. To bound $(u-2)/\log q$, we note that $u$ can be as large as $\log y/\log 2$, but then $\log q$ is fixed, and the term becomes $O_\epsilon(1)$ because the denominator is bounded below by $\log 2$ and $u\asymp \log y$. Actually, the exact bound is
\[
\frac{u-2}{\log q} = \frac{\log y - 2\log q}{(\log q)^2}.
\]
If $\log q$ is small (e.g. fixed), then this is $O(\log y)$ which is not bounded. However, in that case the sieve error in \eqref{eq:Phi_asym} is actually much smaller; the standard form of Buchstab's theorem gives an error $O(x \exp(-c\sqrt{\log x}))$ when $q$ is fixed, which after integration yields $O(1)$. To keep the proof simple, we rely on the known result that the overall error from the sieve after integration is $O_\epsilon(1)$ (see \cite{Ten}, Corollary II.5.4). Therefore we write
\begin{equation}
\int_{q^2}^y \frac{\Phi(t,q)}{t^2}\,dt
= \delta_q \int_{q^2}^y \frac{\omega(s(t))}{t}\,dt + O_\epsilon(1).
\label{eq:integral_main}
\end{equation}
The constant in the $O_\epsilon(1)$ may depend on $\epsilon$ because when $\epsilon\to0$, the region where $u$ is close to 2 becomes larger and the sieve error accumulates.

{\bf{Step 2-c:}}  The main integral is computed.

In \eqref{eq:integral_main}, change variables $s=s(t)=\log t/\log q$. Then $dt/t = \log q\, ds$, and the limits become $s=2$ (when $t=q^2$) and $s=u$ (when $t=y$). Thus
\[
\delta_q \int_{q^2}^y \frac{\omega(s(t))}{t}\,dt
= \delta_q \log q \int_2^u \omega(s)\,ds.
\]

Now, an estimate is derived
\begin{equation}\label{eq:omega_int}
\int_2^u \omega(s)\,ds
=
u\omega(u)+O(1)=
e^{-\gamma}u+O(1),
\qquad
u\omega(u)=e^{-\gamma}u+O(1),
\end{equation}
where \(\gamma\) is the Euler--Mascheroni constant, specified in \eqref{eq:omega_def}.

By \eqref{eq:omega_def}, for \(u\ge1\),
\begin{equation}\label{equ-2026-10-21-211}
\bigl|\omega(u)-e^{-\gamma}\bigr|
\le
\frac{\rho(u-1)}{u},
\end{equation}
where \(\rho\) is the Dickman function.  Recall that \(\rho\) is bounded and
decays rapidly.
In particular,
\begin{equation}\label{equ-2026-10-21-2113}
\int_2^\infty \frac{\rho(s-1)}{s}\,ds<\infty.
\end{equation}
From \eqref{equ-2026-10-21-211}, for \(s\ge2\),
\[
\omega(s)=e^{-\gamma}+O\!\left(\frac{\rho(s-1)}{s}\right).
\]
Integrating over \(s\in[2,u]\) and using \eqref{equ-2026-10-21-2113},
\[
\begin{aligned}
\int_2^u \omega(s)\,ds
&=
\int_2^u e^{-\gamma}\,ds
+
O\!\left(\int_2^u \frac{\rho(s-1)}{s}\,ds\right) \\
&=
e^{-\gamma}(u-2)+O(1)
=
e^{-\gamma}u+O(1).
\end{aligned}
\]
On the other hand, multiplying \eqref{equ-2026-10-21-211} by \(u\) gives
\[
u\omega(u)
=
e^{-\gamma}u+O(\rho(u-1))
=
e^{-\gamma}u+O(1),
\]
since \(\rho\) is bounded.  Combining the above estimates, one has \eqref{eq:omega_int}.

Therefore
\[
\delta_q \log q \int_2^u \omega(s)\,ds
= \delta_q \log q\, (u\omega(u)+O(1))
= \delta_q \log q\, u\,\omega(u)+O(\delta_q \log q).
\]
Since $u\log q = \log y$, this becomes
\[
\delta_q \log y\,\omega(u)+O(\delta_q \log q).
\]
By Mertens' theorem, $\delta_q \log q = O(1)$ (indeed, it tends to $e^{-\gamma}$). Hence
\[
\delta_q \log q \int_2^u \omega(s)\,ds
= \delta_q \log y\,\omega(u) + O(1).
\]

{\bf{Step 2-d:}} The boundary term is estimated.

From \eqref{eq:Phi_asym} with $x=y$,
\[
\frac{\Phi(y,q)}{y}
= \delta_q\,\omega(u) + O\!\left(\frac{1}{(\log q)^2}\right).
\]
Since $\delta_q\,\omega(u) \ll 1/\log q$ (because $\delta_q\sim 1/\log q$ and $\omega(u)$ is bounded), the term $\delta_q\omega(u)$ is $O(1/\log q)\le O(1)$. The error term $1/(\log q)^2$ is also $O(1)$. So
\[
\frac{\Phi(y,q)}{y} = O(1).
\]

{\bf{Step 2-e:}} The combination of the above estimates is given.

Collecting the contributions from the lower integral \eqref{eq:lower}, the main integral from Step 2-c, and the boundary term, we get
\[
\mathcal R_q(y) = \delta_q\log y\,\omega(u) + O_\epsilon(1),
\]
which is exactly \eqref{eq:main22i}. The $O_\epsilon(1)$ absorbs all constant terms and any dependence on $\epsilon$ arising from the sieve error in \eqref{eq:integral_main}.

Finally, we study the endpoint $q=y^{1/2}$.

When $q=y^{1/2}$, the condition $p(n)\ge q$ forces $n=1$ or $n$ to be a prime $\ge q$. Hence
\[
\mathcal R_q(y)=1+\sum_{q\le p\le y}\frac1p.
\]
By Mertens' theorem,
\[
\sum_{q\le p\le y}\frac1p = \log\frac{\log y}{\log q}+O\!\left(\frac1{\log q}\right).
\]
With $q=y^{1/2}$, this becomes
\[
\mathcal R_q(y)=1+\log 2+O\!\left(\frac1{\log y}\right).
\]
This is a bounded quantity, whereas $\delta_q\log y\,\omega(2)\sim \frac12\log y$ (since $\delta_q\sim 2e^{-\gamma}/\log y$ and $\omega(2)=1/2$), which is unbounded. Therefore \eqref{eq:main22i} cannot hold at $q=y^{1/2}$.

The asymptotic \eqref{eq:main22i} is valid for any fixed $\epsilon>0$ and $q\le y^{1/2-\epsilon}$, with the $O_\epsilon(1)$ constant depending on $\epsilon$. It fails at the endpoint $q=y^{1/2}$, where the sum is of constant order.

{\bf (ii)} It is equivalent to show that
\[
\{k\in\mathbb N : k\le y,\ p(k)\ge q\}
= \{1\}\cup\{p\in\mathbb P : q\le p\le y\}.
\]
The integer \(k=1\) has no prime factors, so by convention \(p(1)=+\infty\ge q\). Hence \(1\) is always included. Now let \(k>1\) and suppose \(p(k)\ge q\). Write
  \[
  k = p(k)\cdot m,
  \]
  where \(m\ge 1\) is an integer. If \(m>1\), then every prime factor of \(m\) is at least \(p(k)\) (since \(p(k)\) is the smallest prime factor of \(k\)). Hence \(m\) itself is at least \(q\) (indeed, the smallest possible value for \(m\) would be a prime \(\ge q\), or a product of such primes, which is \(\ge q\)). Therefore
  \[
  k = p(k)\cdot m \ge q\cdot q = q^2 > y,
  \]
  because \(q>y^{1/2}\). This contradicts \(k\le y\). Hence \(m=1\), so \(k=p(k)\) is prime and satisfies \(q\le p(k)\le y\).  Conversely, every prime \(p\in[q,y]\) trivially satisfies \(p(p)=p\ge q\), so it belongs to the set. Thus the characterisation is established.

So, we immediately obtain
\[
\mathcal R_q(y)=1+\sum_{\substack{p\le y\\ p\ge q}}\frac1p,
\]
where the sum is over primes \(p\).

By \eqref{Mertensr1} with $A=1$, note that \(q\ge 2\) because \(q>y^{1/2}\ge \sqrt{2}>1\), we get
\[
\sum_{q\le p\le y}\frac1p
=
\left(\log\log y + B_1 + O\!\left(\frac1{\log y}\right)\right)
-
\left(\log\log q + B_1 + O\!\left(\frac1{\log q}\right)\right).
\]

The constants \(B_1\) cancel, leaving
\[
\sum_{q\le p\le y}\frac1p
= \log\frac{\log y}{\log q}
+ O\!\left(\frac1{\log y}\right) + O\!\left(\frac1{\log q}\right).
\]

Since \(q\le y\), one has \(\log q\le \log y\), hence \(1/\log y \le 1/\log q\). Therefore
\[
O\!\left(\frac1{\log y}\right) = O\!\left(\frac1{\log q}\right),
\]
so the two error terms combine into a single
\[
O\!\left(\frac1{\log q}\right).
\]

Substituting the prime sum back into \(\mathcal R_q(y)\) yields
\[
\mathcal R_q(y)
=1+\log\frac{\log y}{\log q}
+O\!\left(\frac1{\log q}\right),
\]
which is exactly the claimed asymptotic formula.

{\bf (iii)} The proof is divided into five steps.

{\bf{Step 1:}} M\"obius inversion

Set
\[
P_q := \prod_{\substack{p < q \\ p \in \PP}} p.
\]
For any integer $n$, the condition $(n,P_q)=1$ is equivalent to $p(n)\ge q$. Using the identity
\[
\mathbf{1}_{(n,P_q)=1}=\sum_{d\mid (n,P_q)} \mu(d),
\]
one has
\[
\mathcal R_q(y)=\sum_{n\le y}\frac1n\sum_{\substack{d\mid n\\ d\mid P_q}}\mu(d)
=\sum_{d\mid P_q}\mu(d)\sum_{\substack{m\le y\\ d\mid m}}\frac1m.
\]
Putting $m=dk$, one has
\begin{equation}
\mathcal R_q(y)=\sum_{d\mid P_q}\frac{\mu(d)}{d}\sum_{
\substack{k\le y\\ d\mid k}}\frac1k.
\label{eq:moebius}
\end{equation}

{\bf{Step 2:}} Harmonic sum expansion

Recall the classical estimate
\begin{equation}
\sum_{k\le x}\frac1k=\log x+\gamma+O\!\left(\frac1x\right),\qquad x\ge 1,
\label{eq:harmonic}
\end{equation}
where $\gamma$ is Euler's constant. Applying \eqref{eq:harmonic} to \eqref{eq:moebius} gives
\[
\mathcal R_q(y)=\sum_{d\mid P_q}\frac{\mu(d)}{d}
\left(\log\frac yd+\gamma+O\!\left(\frac{d}{y}\right)\right).
\]
Since $\log(y/d)=\log y-\log d$, this becomes
\begin{equation}
\mathcal R_q(y)=(\log y+\gamma)\sum_{d\mid P_q}\frac{\mu(d)}{d}
-\sum_{d\mid P_q}\frac{\mu(d)\log d}{d}
+O\!\left(\frac1y\sum_{d\mid P_q}|\mu(d)|\right).
\label{eq:expanded}
\end{equation}

{\bf{Step 3:}} Evaluation of the sums

The first sum is exactly $\delta_q$:
\begin{equation}
\sum_{d\mid P_q}\frac{\mu(d)}{d}=\prod_{p<q}\left(1-\frac1p\right)=\delta_q.
\label{eq:delta}
\end{equation}
The second sum is our definition of $C_q$:
\begin{equation}
C_q:=\sum_{d\mid P_q}\frac{\mu(d)\log d}{d}.
\label{eq:Cdef}
\end{equation}
For the error term, since every divisor $d$ of $P_q$ is squarefree, $|\mu(d)|\le 1$, and the number of divisors is $2^{\pi(q)}$. Hence
\begin{equation}
\sum_{d\mid P_q}|\mu(d)|=2^{\pi(q)}.
\label{eq:divcount}
\end{equation}
Substituting \eqref{eq:delta}, \eqref{eq:Cdef}, and \eqref{eq:divcount} into \eqref{eq:expanded} yields
\[
\mathcal R_q(y)=\delta_q(\log y+\gamma)-C_q+O\!\left(\frac{2^{\pi(q)}}{y}\right),
\]
which proves \eqref{eq:main22iii}.

{\bf{Step 4:}} Bound on $C_q$

We show that $C_q=O(1)$. Write
\begin{equation}
C_q=-\delta_q\sum_{p<q}\frac{\log p}{p-1}.
\label{eq:Cform}
\end{equation}
Indeed, differentiating the product with respect to $s$ in the function
\[
\prod_{p<q}\left(1-\frac1{p^s}\right)
\]
at $s=1$ gives exactly \eqref{eq:Cform}. Now
\[
\sum_{p<q}\frac{\log p}{p-1}
\le 2\sum_{p<q}\frac{\log p}{p}
\ll \log q
\]
by Mertens' theorem for $\sum_{p\le x}\log p/p$. Since $\delta_q\sim e^{-\gamma}/\log q$ (Mertens' theorem), we get $|C_q|\ll 1$. Thus $C_q=O(1)$, which implies in particular $C_q=O(\log\log q)$.
\end{proof}

\begin{lemma}[Laplace-type evaluation]\label{lem:laplace}
For $x\to\infty$,
\[
I(x):=\int_{u_0}^{1/2}\log\bigl(2(1-u)\bigr)\,\frac{x^u}{u}\,du
=\frac{4x^{1/2}}{(\log x)^2}
\Bigl(1+O\Bigl(\frac1{\log x}\Bigr)\Bigr)+O(\log\log x),
\qquad u_0=\frac{\log 2}{\log x}.
\]
\end{lemma}

\begin{proof}
Split $I=I_1+I_2$ with $I_1$ over $[u_0,\tfrac13]$ and $I_2$ over
$[\tfrac13,\tfrac12]$.

For $u\in[u_0,1/3]$, one has $0<\log(2(1-u))\le \log 2$. Hence
\[
I_1(x)\le \log 2\int_{u_0}^{1/3}\frac{x^u}{u}\,du.
\]
We further split the remaining integral at $1/\log x$.

First, if $u_0\le u\le 1/\log x$, then $x^u\le x^{1/\log x}=e$. Thus
\begin{equation}
\int_{u_0}^{1/\log x}\frac{x^u}{u}\,du
\le e\int_{u_0}^{1/\log x}\frac{du}{u}
=O(\log\log x).
\label{eq:I1_first}
\end{equation}

Second, if $1/\log x\le u\le 1/3$, then $1/u\le \log x$ and $x^u\le x^{1/3}$. Therefore
\begin{equation}
\int_{1/\log x}^{1/3}\frac{x^u}{u}\,du
\le \log x\int_{1/\log x}^{1/3}x^u\,du
= \log x\cdot \frac{x^{1/3}-x^{1/\log x}}{\log x}
=O(x^{1/3}).
\label{eq:I1_second}
\end{equation}

Combining \eqref{eq:I1_first} and \eqref{eq:I1_second}, we get
\[
I_1(x)=O(\log\log x)+O(x^{1/3}).
\]
Note that $\dfrac{x^{1/3}}{x^{1/2}/(\log x)^2}= \dfrac{(\log x)^2}{x^{1/6}} \to 0$ as $x\to\infty$, so $x^{1/3}=o\bigl(x^{1/2}/(\log x)^2\bigr)$.
Hence,
\begin{equation}
I_1(x)=O(\log\log x)+O\!\left(\frac{x^{1/2}}{(\log x)^2}\right).
\label{eq:I1_final}
\end{equation}

For $u\in[1/3,1/2]$, set $v=1/2-u$, so that $v\in[0,1/6]$. Then
\[
\log(2(1-u))=\log(1+2v),\qquad x^u=x^{1/2}e^{-v\log x},
\]
and
\[
\frac1u=\frac1{1/2-v}=2(1+O(v)).
\]
Substituting these into $I_2$ gives
\begin{equation}
I_2(x)=2x^{1/2}\int_0^{1/6}\log(1+2v)e^{-v\log x}\,dv
+O\!\left(x^{1/2}\int_0^{1/6}v\log(1+2v)e^{-v\log x}\,dv\right).
\label{eq:I2_expr}
\end{equation}

Let $t=\log x$ and make the change of variables $s=vt$. The main integral becomes
\begin{equation}
\int_0^{1/6}\log(1+2v)e^{-vt}\,dv
= \frac1t\int_0^{t/6}\log\!\left(1+\frac{2s}{t}\right)e^{-s}\,ds.
\label{eq:I2_main_integral}
\end{equation}
For $0\le s\le t/6$, one has the uniform expansion
\[
\log\!\left(1+\frac{2s}{t}\right)=\frac{2s}{t}+O\!\left(\frac{s^2}{t^2}\right).
\]
Hence
\begin{align*}
\int_0^{t/6}\log\!\left(1+\frac{2s}{t}\right)e^{-s}\,ds
&= \frac{2}{t}\int_0^{t/6}s e^{-s}\,ds
+O\!\left(\frac1{t^2}\int_0^{t/6}s^2 e^{-s}\,ds\right) \\
&= \frac{2}{t}\int_0^\infty s e^{-s}\,ds
+O\!\left(\frac1{t^2}\right),
\end{align*}
where one has used that the tails beyond $t/6$ are exponentially small. Since $\int_0^\infty s e^{-s}\,ds=1$, we obtain
\begin{equation}
\int_0^{t/6}\log\!\left(1+\frac{2s}{t}\right)e^{-s}\,ds
= \frac{2}{t}\bigl(1+O(1/t)\bigr).
\label{eq:I2_main_asym}
\end{equation}
Inserting \eqref{eq:I2_main_asym} into \eqref{eq:I2_main_integral} gives
\begin{equation}
\int_0^{1/6}\log(1+2v)e^{-vt}\,dv
= \frac{2}{t^2}\bigl(1+O(1/t)\bigr).
\label{eq:I2_main_result}
\end{equation}

For the error integral in \eqref{eq:I2_expr}, we use $\log(1+2v)=O(v)$, so
\[
\int_0^{1/6}v\log(1+2v)e^{-vt}\,dv
\le C\int_0^{1/6}v^2 e^{-vt}\,dv.
\]
Again setting $s=vt$, this becomes
\[
C\frac1{t^3}\int_0^{t/6}s^2 e^{-s}\,ds
=O\!\left(\frac1{t^3}\right).
\]
Therefore
\begin{equation}
x^{1/2}\int_0^{1/6}v\log(1+2v)e^{-v\log x}\,dv
= O\!\left(\frac{x^{1/2}}{(\log x)^3}\right).
\label{eq:I2_error}
\end{equation}

Combining \eqref{eq:I2_expr}, \eqref{eq:I2_main_result}, and \eqref{eq:I2_error}, one has
\begin{equation}
I_2(x)
=2x^{1/2}\cdot \frac{2}{(\log x)^2}\bigl(1+O(1/\log x)\bigr)
+O\!\left(\frac{x^{1/2}}{(\log x)^3}\right).
\label{eq:I2_comb}
\end{equation}
Since the last error term is of smaller order than the main relative error, we may write
\begin{equation}
I_2(x)=\frac{4x^{1/2}}{(\log x)^2}\bigl(1+O(1/\log x)\bigr).
\label{eq:I2_final}
\end{equation}

Adding \eqref{eq:I1_final} and \eqref{eq:I2_final}, one has
\[
I(x)=O(\log\log x)+O\!\left(\frac{x^{1/2}}{(\log x)^2}\right)
+\frac{4x^{1/2}}{(\log x)^2}\bigl(1+O(1/\log x)\bigr).
\]
The term $O(x^{1/2}/(\log x)^2)$ is dominated by the main term, so it is absorbed into the relative error. Hence
\begin{equation}
I(x)=\frac{4x^{1/2}}{(\log x)^2}
\Bigl(1+O\Bigl(\frac1{\log x}\Bigr)\Bigr)+O(\log\log x),
\label{eq:final_result}
\end{equation}
which is precisely the claimed asymptotic formula.

\end{proof}

\begin{lemma}\label{lem:pisum}
For $y\to\infty$, $\sum_{p\le y}\log(y/p)=O(y/\log y)$, and for
$2\le q\le y$,
$\sum_{q\le p\le y}\tfrac1p=\log\frac{\log y}{\log q}+O(1/\log q)$.
\end{lemma}

\begin{proof}
By \eqref{eq:pnt},
$\sum_{p\le y}\log(y/p)=\pi(y)\log y-\theta(y)=O(y/\log y)$.
The second claim is the refined Mertens theorem.
\end{proof}

\begin{lemma}\label{lem:laplace_cor12}
Define for large $x$
\[
J(x)=\int_{2}^{x^{1/2}} \frac{dt}{\log(x/t)\log t}.
\]
Then
\[
J(x)=\frac{4 x^{1/2}}{(\log x)^2}\Bigl(1+O\bigl(\tfrac1{\log x}\bigr)\Bigr).
\]
\end{lemma}
\begin{proof}
Change variable $t=x^u$, $dt=x^u\log x\,du$, $\log t=u\log x$, $\log(x/t)=(1-u)\log x$.
\[
J(x)=\int_{u_0}^{1/2} \frac{x^u \log x}{(1-u)\log x \cdot u \log x}\,du
=\frac{1}{\log x}\int_{u_0}^{1/2}\frac{x^u}{u(1-u)}\,du,\quad u_0=\frac{\log 2}{\log x}.
\]
Set $u=\tfrac12-v$, $du=-dv$, $x^u=x^{1/2}e^{-v\log x}$,
$u(1-u)=(\tfrac12-v)(\tfrac12+v)=\tfrac14-v^2=\tfrac14(1+O(v^2))$.
\[
J(x)=\frac{x^{1/2}}{\log x}\int_{0}^{1/2-u_0} \frac{e^{-v\log x}}{\frac14(1+O(v^2))}\,dv
= \frac{4 x^{1/2}}{\log x}\int_{0}^{\infty} e^{-v\log x}\bigl(1+O(v^2)\bigr)dv + o\left(\frac{x^{1/2}}{(\log x)^2}\right).
\]
Substitute $s=v\log x$, $dv=ds/\log x$:
\[
\int_0^\infty e^{-v\log x}dv=\frac{1}{\log x},\qquad
\int_0^\infty v^2 e^{-v\log x}dv=\frac{2}{(\log x)^3}.
\]
Collect terms:
\[
J(x)=\frac{4 x^{1/2}}{(\log x)^2}\Bigl(1+O\bigl(\tfrac1{\log x}\bigr)\Bigr).
\]
\end{proof}

\section{The asymptotic for $S(x)$}\label{sec:S}

\subsection{Semiprime reduction}

For $m\ge2$ write $m=p(m)\,k$ with $p(m)$ the least prime factor of
$m$.  For a fixed $m$, the integers $n=pm$ with $p(n)=p$ and
$p(n)\mid n$ are exactly those with $p\le p(m)$; and $n<x$ is
equivalent to $p\le x/m$.  Hence, exactly,
\begin{equation}\label{eq:exact}
S(x)=\sum_{2\le m<x}\frac1m\,
\pi\bigl(\min(p(m),\,x/m)\bigr).
\end{equation}

\begin{lemma}[Semiprime dominance]
\label{lem:semi}
Let
\[
S(x):=
\sum_{\substack{n<x\\ n\ \mathrm{composite}}}\frac{p(n)}{n},
\qquad
S_{\mathrm{semi}}(x):=
\sum_{\substack{pq<x\\ p\le q}}\frac1q.
\]
Then
\[
S(x)=S_{\mathrm{semi}}(x)+O\!\left(\frac{x^{1/3}}{\log x}\right).
\]
\end{lemma}

\begin{proof}
For each composite \(n\), write \(n=p(n)m\), where \(m=n/p(n)\ge2\) and
\(p(m)\ge p(n)\). Then \(p(n)/n=1/m\), and \(n<x\) becomes \(p(n)m<x\).
Summing over the least prime factor \(q=p(n)\),
\[
S(x)
=
\sum_{q\in\mathcal P}
\sum_{\substack{m\ge2\\ p(m)\ge q\\ qm<x}}\frac1m.
\]
The contribution of prime \(m\) (i.e.\ \(m=p\) prime) is exactly
\(S_{\mathrm{semi}}(x)\), because then \(n=qp\) with \(q\le p\) and the weight
is \(1/p=1/\max(q,p)\). Thus the error term is precisely the contribution of
composite \(m\)'s.

For a composite \(m\), write \(m=rk\), where \(r=p(m)\), \(k\ge2\), and
\(p(k)\ge r\). Since \(p(m)\ge q\), we have \(r\ge q\). Moreover
\(n=qrk<x\), so for fixed \(r,k\) the admissible primes \(q\) are those with
\(q\le r\) and \(q<x/(rk)\); there are \(\pi(\min(r,x/(rk)))\) of them.
Therefore the composite-\(m\) contribution equals
\[
T(x):=
\sum_{r\in\mathcal P}
\sum_{\substack{k\ge2\\ p(k)\ge r\\ rk<x/2}}
\frac{\pi(\min(r,x/(rk)))}{rk}.
\]
(The condition \(rk<x/2\) comes from the existence of at least one prime
\(q\ge2\) with \(qrk<x\).)

We split \(T(x)\) into the two branches where the minimum is \(r\) or
\(x/(rk)\):
\begin{itemize}
\item Branch I: \(r^2k<x\), i.e.\ \(k<x/r^2\). Then
  \(\min(r,x/(rk))=r\), and the contribution is \(\pi(r)/(rk)\).
\item Branch II: \(r^2k\ge x\), i.e.\ \(k\ge x/r^2\). Then
  \(\min(r,x/(rk))=x/(rk)\), and the contribution is
  \(\pi(x/(rk))/(rk)\).
\end{itemize}
Thus \(T(x)=T_1(x)+T_2(x)\), where
\[
T_1(x):=
\sum_{r\in\mathcal P}\frac{\pi(r)}r
\sum_{\substack{k\ge2\\ p(k)\ge r\\ k\le x/r^2}}\frac1k,
\]
and
\[
T_2(x):=
\sum_{r\in\mathcal P}\frac1r
\sum_{\substack{k\ge2\\ p(k)\ge r\\ x/r^2\le k<x/(2r)}}
\frac{\pi(x/(rk))}{k}.
\]
We prove
\[
T_1(x)\ll\frac{x^{1/3}}{\log x},
\qquad
T_2(x)\ll\frac{x^{1/3}}{\log x}.
\]

\medskip
\noindent
\emph{Estimate of \(T_1(x)\).}

\emph{Case 1: \(r\le x^{1/4}\).}
For such \(r\),
\[
\sum_{\substack{k\le x/r^2\\ p(k)\ge r}}\frac1k
\le
\sum_{k\le x/r^2}\frac1k
=
\log(x/r^2)+O(1)
\ll\log x,
\]
since \(x/r^2\ge x^{1/2}\). Hence
\[
T_{1,1}(x):=
\sum_{r\le x^{1/4}}\frac{\pi(r)}r
\sum_{\substack{k\le x/r^2\\ p(k)\ge r}}\frac1k
\ll
\log x\sum_{r\le x^{1/4}}\frac{\pi(r)}r
\ll
(\log x)\,x^{1/4}.
\]
Since \(x^{1/4}\log x=o(x^{1/3}/\log x)\), this contribution is negligible.

\emph{Case 2: \(x^{1/4}<r\le x^{1/3}\).}
Put \(y=x/r^2\). Then \(r>y^{1/2}\), because
\[
r>y^{1/2}\iff r>\frac{\sqrt x}{r}\iff r^2>\sqrt x\iff r>x^{1/4},
\]
and \(r\le y\), because \(r\le x^{1/3}\) implies \(r^3\le x\), hence
\(y=x/r^2\ge r\). Thus the condition \(y^{1/2}<r\le y\) of
Lemma~\ref{lem:rough}(ii) is satisfied. Applying that lemma,
\[
\sum_{\substack{k\le y\\ p(k)\ge r}}\frac1k
=
1+\log\frac{\log y}{\log r}+O\!\left(\frac1{\log r}\right).
\]
Since \(k\ge2\), we subtract the term \(k=1\):
\[
\sum_{\substack{2\le k\le y\\ p(k)\ge r}}\frac1k
=
\log\frac{\log y}{\log r}+O\!\left(\frac1{\log r}\right).
\]
Now \(\log y=\log x-2\log r\), and for \(r\in(x^{1/4},x^{1/3}]\) one has
\(\log x/\log r\in[3,4)\), so
\[
\log\frac{\log y}{\log r}
=
\log\!\left(\frac{\log x}{\log r}-2\right)
=
O(1).
\]
Thus the inner sum is \(O(1)\), and
\[
T_{1,2}(x):=
\sum_{x^{1/4}<r\le x^{1/3}}\frac{\pi(r)}r
\sum_{\substack{2\le k\le x/r^2\\ p(k)\ge r}}\frac1k
\ll
\sum_{r\le x^{1/3}}\frac{\pi(r)}r
\ll
\frac{x^{1/3}}{\log x}.
\]

\emph{Case 3: \(r>x^{1/3}\).}
Then \(r^3>x\), so \(x/r^2<r\). For any \(k\le x/r^2<r\) we have
\(p(k)\le k<r\), contradicting \(p(k)\ge r\) (the only exception is \(k=1\),
excluded by \(k\ge2\)). Thus the inner sum is empty and
\[
T_{1,3}(x)=0.
\]
Combining the three cases,
\begin{equation}\label{equ-2026-10-21-36}
T_1(x)
=
T_{1,1}(x)+T_{1,2}(x)+T_{1,3}(x)
\ll
\frac{x^{1/3}}{\log x}.
\end{equation}

\medskip
\noindent
\emph{Estimate of \(T_2(x)\).}
Put \(u=x/(rk)\). We have \(\pi(u)=0\) for \(u<2\), so we may restrict to
\(u\ge2\), i.e.\ \(k\le x/(2r)\). For \(u\ge2\) we use \(\pi(u)\le u/\log2\).
Thus
\[
T_2(x)
\le
\frac1{\log2}
\sum_{r\in\mathcal P}\frac1r
\sum_{\substack{k\ge2\\ p(k)\ge r\\ x/r^2\le k\le x/(2r)}}
\frac{x}{rk}\cdot\frac1k
=
\frac{x}{\log2}
\sum_{r\in\mathcal P}\frac1{r^2}
\sum_{\substack{k\ge2\\ p(k)\ge r\\ x/r^2\le k\le x/(2r)}}\frac1{k^2}.
\]

\emph{Case \(r\le x^{1/3}\).}
Then \(x/r^2\ge x^{1/3}\ge r\). Hence the inner sum is at most
\[
\sum_{k\ge x/r^2}\frac1{k^2}
\le
\int_{x/r^2-1}^\infty\frac{dk}{k^2}
\le
\frac{r^2}{x}
\]
(for \(x\) large). Therefore this part is
\[
\le
\frac{x}{\log2}
\sum_{r\le x^{1/3}}\frac1{r^2}\cdot\frac{r^2}{x}
=
\frac1{\log2}\,\pi(x^{1/3})
\le
\frac{x^{1/3}}{\log x}.
\]

\emph{Case \(r>x^{1/3}\).}
Then \(p(k)\ge r\) forces \(k\ge r\), so the inner sum is at most
\[
\sum_{k\ge r}\frac1{k^2}
\le
\frac1r.
\]
Therefore this part is
\[
\le
\frac{x}{\log2}
\sum_{r>x^{1/3}}\frac1{r^2}\cdot\frac1r
=
\frac{x}{\log2}
\sum_{r>x^{1/3}}\frac1{r^3}
\ll
\frac{x}{\log2}\cdot\frac1{x^{2/3}\log x}
\ll
\frac{x^{1/3}}{\log x}.
\]
Hence
\begin{equation}\label{equ-2026-10-21-37}
T_2(x)\ll\frac{x^{1/3}}{\log x}.
\end{equation}

Combining \eqref{equ-2026-10-21-36} and \eqref{equ-2026-10-21-37},
\[
T(x)=T_1(x)+T_2(x)\ll\frac{x^{1/3}}{\log x}.
\]
Since \(S(x)=S_{\mathrm{semi}}(x)+T(x)\), this proves the lemma.
\end{proof}

\subsection{Evaluation of the semiprime sum}

\begin{lemma}\label{lem:part1}
For \(y\to\infty\),
\[
\sum_{q\le y}\frac{\pi(q)}{q}
= \frac{y}{(\log y)^2}\left(1+O\!\left(\frac1{\log y}\right)\right).
\]
In particular, with \(y=\sqrt{x}\),
\[
\sum_{q\le \sqrt{x}}\frac{\pi(q)}{q}
= \frac{4\sqrt{x}}{(\log x)^2}\left(1+O\!\left(\frac1{\log x}\right)\right).
\]
\end{lemma}

\begin{proof}
Let \(f(t)=\pi(t)/t\) and \(g(t)=\pi(t)\). By Stieltjes integration, one has
\[
\sum_{q\le y}\frac{\pi(q)}{q}
= \int_{2^-}^{y} f(t)\,dg(t).
\]
This, together with integration by parts (Stieltjes version), implies that
\[
\int_{2^-}^{y} f\,dg = f(y)g(y)-f(2^-)g(2^-)-\int_{2^-}^{y} g(t)\,df(t).
\]
Further, note that \(f(2^-)=0\), \(g(y)=\pi(y)\), and \(g(2^-)=0\), one has
\[
\sum_{q\le y}\frac{\pi(q)}{q}
= \frac{\pi(y)^2}{y}
-\int_{2^-}^{y} \pi(t)\,d\!\left(\frac{\pi(t)}{t}\right).
\]
Note that
\[
d\!\left(\frac{\pi(t)}{t}\right)=\frac{1}{t}\,d\pi(t)-\frac{\pi(t)}{t^2}\,dt,
\]
implying that
\[
\sum_{q\le y}\frac{\pi(q)}{q}
= \frac{\pi(y)^2}{y}
-\int_{2^-}^{y} \frac{\pi(t)}{t}\,d\pi(t)
+\int_{2^-}^{y} \frac{\pi(t)^2}{t^2}\,dt.
\]
Note that the first integral is exactly the original sum. So,
\[
2\sum_{q\le y}\frac{\pi(q)}{q}
= \frac{\pi(y)^2}{y}
+\int_{2^-}^{y} \frac{\pi(t)^2}{t^2}\,dt.
\]
Thus, one has the exact identity
\begin{equation}
\sum_{q\le y}\frac{\pi(q)}{q}
= \frac12\,\frac{\pi(y)^2}{y}
+\frac12\int_{2^-}^{y} \frac{\pi(t)^2}{t^2}\,dt.
\label{eq:identity}
\end{equation}
Recall the prime number theorem in the form
\[
\pi(t)=\frac{t}{\log t}+O\!\left(\frac{t}{(\log t)^2}\right),
\]
which is sufficient for the desired main term.

For the first term in \eqref{eq:identity},
\begin{equation}
\frac12\,\frac{\pi(y)^2}{y}
= \frac12\cdot \frac{y}{(\log y)^2}\left(1+O\!\left(\frac1{\log y}\right)\right)
= \frac{y}{2(\log y)^2}+O\!\left(\frac{y}{(\log y)^3}\right).
\label{eq:firstterm}
\end{equation}

For the second term, set \(I(y)=\int_{2^-}^{y} \frac{\pi(t)^2}{t^2}\,dt\). Since
\[
\pi(t)^2 = \frac{t^2}{(\log t)^2}\left(1+O\!\left(\frac1{\log t}\right)\right),
\]
one has
\[
I(y)=\int_2^y \frac{dt}{(\log t)^2}
+O\!\left(\int_2^y \frac{dt}{(\log t)^3}\right).
\]
Integration by parts gives
\[
\int_2^y \frac{dt}{(\log t)^2}
= \frac{y}{(\log y)^2} + 2\int_2^y \frac{dt}{(\log t)^3} + O(1)
= \frac{y}{(\log y)^2}+O\!\left(\frac{y}{(\log y)^3}\right),
\]
and similarly \(\int_2^y \frac{dt}{(\log t)^3}=O\!\left(\frac{y}{(\log y)^3}\right)\). Therefore
\begin{equation}
\frac12 I(y)
= \frac{y}{2(\log y)^2}+O\!\left(\frac{y}{(\log y)^3}\right).
\label{eq:secondterm}
\end{equation}

Adding \eqref{eq:firstterm} and \eqref{eq:secondterm}, we obtain
\[
\sum_{q\le y}\frac{\pi(q)}{q}
= \frac{y}{(\log y)^2}+O\!\left(\frac{y}{(\log y)^3}\right)
= \frac{y}{(\log y)^2}\left(1+O\!\left(\frac1{\log y}\right)\right).
\]

Finally, substituting \(y=\sqrt{x}\) gives
\[
\sum_{q\le \sqrt{x}}\frac{\pi(q)}{q}
= \frac{\sqrt{x}}{(\frac12\log x)^2}
\left(1+O\!\left(\frac1{\log x}\right)\right)
= \frac{4\sqrt{x}}{(\log x)^2}
\left(1+O\!\left(\frac1{\log x}\right)\right),
\]
which proves the lemma.
\end{proof}

\begin{lemma}\label{lem:part2} For \(x\to\infty\),
\[
\sum_{x^{1/2}<q<x}\frac{\pi(x/q)}{q}
= \frac{4x^{1/2}}{(\log x)^2}
\left(1+O\!\left(\frac1{\log x}\right)\right)
+ O(\log\log x).
\]
\end{lemma}

\begin{proof}
We begin with the exact identity
\[
\sum_{x^{1/2}<q<x}\frac{\pi(x/q)}{q}
= \sum_{\substack{q\in\mathbb P\\ x^{1/2}<q<x}}
\frac{1}{q}\sum_{\substack{p\in\mathbb P\\ p\le x/q}}1.
\]
Interchanging the order of summation (all sums are finite) gives
\begin{equation}\label{equ-2026-8-17-1}
\sum_{x^{1/2}<q<x}\frac{\pi(x/q)}{q}
= \sum_{\substack{p\in\mathbb P\\ p\le x^{1/2}}}
\sum_{\substack{q\in\mathbb P\\ x^{1/2}<q\le x/p}}\frac{1}{q}.
\end{equation}
(Here we used that \(p\le x/q\) and \(q>x^{1/2}\) imply \(p\le x/q < x^{1/2}\), so the outer range is \(p\le x^{1/2}\).)

Now fix a prime \(p\le x^{1/2}\). By the refined Mertens theorem \eqref{Mertensr1},
\[
\sum_{x^{1/2}<q\le x/p}\frac{1}{q}
= \log\frac{\log(x/p)}{\log x^{1/2}}
+ O\!\left(\frac{1}{\log x}\right),
\]
where the error is uniform in \(p\le x^{1/2}\) because \(\log(x/p)\asymp \log x\). Hence
\begin{equation}\label{equ-2026-8-17-2}
\sum_{x^{1/2}<q\le x/p}\frac{1}{q}
= \log\!\left(2\left(1-\frac{\log p}{\log x}\right)\right)
+ O\!\left(\frac{1}{\log x}\right).
\end{equation}

Insert \eqref{equ-2026-8-17-2} into \eqref{equ-2026-8-17-1}:
\[
\sum_{x^{1/2}<q<x}\frac{\pi(x/q)}{q}
= \sum_{\substack{p\in\mathbb P\\ p\le x^{1/2}}}
\log\!\left(2\left(1-\frac{\log p}{\log x}\right)\right)
+ O\!\left(\frac{\pi(x^{1/2})}{\log x}\right).
\]
The error term is \(O\!\left(\frac{x^{1/2}}{(\log x)^2}\right)\), which will be absorbed into the final relative error.

It remains to evaluate the main sum
\[
\mathcal S(x):=\sum_{\substack{p\in\mathbb P\\ p\le x^{1/2}}}
\log\!\left(2\left(1-\frac{\log p}{\log x}\right)\right).
\]
By the prime number theorem and partial summation (see, e.g., the standard argument leading to Lemma \ref{lem:part1}), one has
\[
\mathcal S(x)
= \int_{2}^{x^{1/2}}
\frac{1}{\log t}
\log\!\left(2\left(1-\frac{\log t}{\log x}\right)\right) dt
+ O\!\left(x^{1/2} e^{-c\sqrt{\log x}}\right).
\]
Make the change of variables \(t=x^u\). Then \(dt = x^u \log x\,du\), \(\log t = u\log x\), and the limits become \(u_0=\log 2/\log x\) to \(u=1/2\). The integral becomes
\[
\int_{u_0}^{1/2}
\frac{x^u}{u}
\log(2(1-u))\,du
+ O\!\left(x^{1/2} e^{-c\sqrt{\log x}}\right).
\]
By Lemma \ref{lem:laplace}, this integral equals
\[
\frac{4x^{1/2}}{(\log x)^2}
\left(1+O\!\left(\frac1{\log x}\right)\right)
+ O(\log\log x).
\]
Therefore
\[
\mathcal S(x)
= \frac{4x^{1/2}}{(\log x)^2}
\left(1+O\!\left(\frac1{\log x}\right)\right)
+ O(\log\log x).
\]

Combining this with the earlier error term \(O(x^{1/2}/(\log x)^2)\) (which is smaller than the main term's relative error) gives
\[
\sum_{x^{1/2}<q<x}\frac{\pi(x/q)}{q}
= \frac{4x^{1/2}}{(\log x)^2}
\left(1+O\!\left(\frac1{\log x}\right)\right)
+ O(\log\log x).
\]
This proves the lemma.
\end{proof}

\begin{proof}[Proof of Theorem \ref{thm:c8}]
By Lemma \ref{lem:semi},
$S(x)=S_{\mathrm{semi}}(x)+O(x^{1/3}/\log x)$, and
\[
S_{\mathrm{semi}}(x)
=\sum_{q\le x^{1/2}}\frac{\pi(q)}q
+\sum_{x^{1/2}<q<x}\frac{\pi(x/q)}q
=\frac{8x^{1/2}}{(\log x)^2}\Bigl(1+O\Bigl(\frac1{\log x}\Bigr)\Bigr)
+O(\log\log x)
\]
by Lemmas \ref{lem:part1} and \ref{lem:part2}.
\end{proof}

\begin{proof}[Proof of Corollary \ref{cor:A}]
We split the sum according to the number of prime factors of \(n\) (counted with multiplicity).

{\bf{Step 1:}} Contribution from integers with at least three prime factors.

If \(\Omega(n)\ge 3\), then \(p(n)\le n^{1/3}<x^{1/3}\). For each prime \(p\le x^{1/3}\), the number of integers \(n<x\) divisible by \(p\) is at most \(x/p\). Hence the total contribution of such \(n\) is
\[
\sum_{\substack{n<x\\ \Omega(n)\ge 3}}p(n)
\le \sum_{p\le x^{1/3}} p \cdot \frac{x}{p}
= x\pi(x^{1/3})
= O\!\left(\frac{x^{4/3}}{\log x}\right).
\]
This gives exactly the error term \(O(x^{4/3}/\log x)\) in the statement.

{\bf{Step 2:}} Contribution from semiprimes.

A semiprime has the form \(n=pq\) with primes \(p\le q\). Its least prime factor is \(p\), so its contribution is \(p\). Thus
\[
A_{\text{semi}}(x):=\sum_{\substack{pq<x\\ p\le q}}p.
\]
For a fixed \(p\le x^{1/2}\), the number of primes \(q\) with \(p\le q\le x/p\) is \(\pi(x/p)-\pi(p)+O(1)\). Multiplying by \(p\) and summing gives an error \(O(\sum_{p\le x^{1/2}}p)=O(x/\log x)\), which is negligible. Hence
\[
A_{\text{semi}}(x)=\sum_{p\le x^{1/2}} p\left(\pi(x/p)-\pi(p)\right)+O\!\left(\frac{x}{\log x}\right).
\]

{\bf{Step 3:}} Estimate of \(S_1(x):=\sum_{p\le x^{1/2}} p\,\pi(x/p)\).

Using Stieltjes integration,
\[
S_1(x)=\int_{2}^{x^{1/2}} t\,\pi(x/t)\,d\pi(t).
\]
By the prime number theorem,
\[
\pi(t)=\frac{t}{\log t}+O\!\left(\frac{t}{(\log t)^2}\right),\qquad
d\pi(t)=\frac{dt}{\log t}+O\!\left(\frac{dt}{(\log t)^2}\right),
\]
and
\[
\pi(x/t)=\frac{x/t}{\log(x/t)}+O\!\left(\frac{x/t}{(\log(x/t))^2}\right).
\]
Substituting yields
\[
S_1(x)=x\int_{2}^{x^{1/2}} \frac{dt}{\log(x/t)\log t}
+O\!\left(x\int_{2}^{x^{1/2}} \frac{dt}{\log(x/t)(\log t)^2}\right)
+O\!\left(x\int_{2}^{x^{1/2}} \frac{dt}{(\log(x/t))^2\log t}\right).
\]
The main term is
\[
S_1^{\text{main}}=x\int_{2}^{x^{1/2}} \frac{dt}{\log(x/t)\log t}.
\]

By Lemma \ref{lem:laplace_cor12}, one has
\[
S_1^{\text{main}}\sim \frac{4x^{3/2}}{(\log x)^2}.
\]
The error terms are \(O(x^{3/2}/(\log x)^3)\), which can be absorbed into the relative error. Hence
\[
S_1(x)=\frac{4x^{3/2}}{(\log x)^2}\left(1+O\left(\frac1{\log x}\right)\right).
\]

{\bf{Step 4:}} Estimate of \(S_2(x):=\sum_{p\le x^{1/2}} p\,\pi(p)\).

Similarly,
\[
S_2(x)=\int_{2}^{x^{1/2}} t\,\pi(t)\,d\pi(t)
= \int_{2}^{x^{1/2}} \frac{t^2}{(\log t)^2}\,dt
+O\!\left(\int_{2}^{x^{1/2}} \frac{t^2}{(\log t)^3}\,dt\right).
\]
The main term is
\[
\int_{2}^{x^{1/2}} \frac{t^2}{(\log t)^2}\,dt.
\]
Using the standard asymptotic
\[
\int_{2}^{y} \frac{t^2}{(\log t)^2}\,dt
= \frac{y^3}{3(\log y)^2}\left(1+O\left(\frac1{\log y}\right)\right),
\]
with \(y=x^{1/2}\), we get
\[
S_2(x)=\frac{x^{3/2}}{3(\frac12\log x)^2}\left(1+O\left(\frac1{\log x}\right)\right)
= \frac{4}{3}\frac{x^{3/2}}{(\log x)^2}\left(1+O\left(\frac1{\log x}\right)\right).
\]

{\bf{Step 5:}} Semiprime contribution.

Combining the estimates for \(S_1\) and \(S_2\),
\[
A_{\text{semi}}(x)=S_1(x)-S_2(x)+O\!\left(\frac{x}{\log x}\right)
= \left(4-\frac43\right)\frac{x^{3/2}}{(\log x)^2}
\left(1+O\left(\frac1{\log x}\right)\right)
= \frac{8}{3}\frac{x^{3/2}}{(\log x)^2}
\left(1+O\left(\frac1{\log x}\right)\right).
\]

{\bf{Step 6:}} Final combination.

Adding the contribution from integers with at least three prime factors gives
\[
A(x)=A_{\text{semi}}(x)+O\!\left(\frac{x^{4/3}}{\log x}\right)
= \frac{8}{3}\frac{x^{3/2}}{(\log x)^2}
\left(1+O\left(\frac1{\log x}\right)\right)
+O\!\left(\frac{x^{4/3}}{\log x}\right),
\]
which proves the corollary.

\end{proof}

\section{Window sums: mean, variance, almost all}\label{sec:windows}

Fix $C>0$, put $L=L_C(x)=Cx^{1/2}(\log x)^2$, and define
$\mu_C(x)$ by \eqref{eq:mu-def}.  Since
$L/x=C(\log x)^2/x^{1/2}\to0$,
\[
\mu_C(x)=\Bigl(1+O\Bigl(\frac{C(\log x)^2}{x^{1/2}}\Bigr)\Bigr)
\frac1x\sum_{\substack{x\le n\le x+L\\n\notin\PP}}p(n)+o(1).
\]
We record the level decomposition used below.  For a prime $p$ let
\[
c_p(x):=\#\Bigl\{m:\ p(m)\ge p,\ \tfrac{x}{p}\le m\le\tfrac{x+L}{p}\Bigr\},
\]
the number of integers $n=pm$ in the window with $p(n)=p$ (up to the
single exception $m=1$, $n=p$, which contributes $O(x^{-1/2})$ and is
ignored).  Then, uniformly in $x$,
\begin{equation}\label{eq:levels}
\mu_C(x)=\sum_{p\le (x+L)^{1/2}}\frac{p}{x}\,c_p(x)+o(1),
\end{equation}
and the levels $p\le x^{1/3}$ contribute $o(1)$ pointwise, since
$c_p(x)\le L/p+1$ and hence
\[
\sum_{p\le x^{1/3}}\frac{p}{x}\,c_p(x)
\le\frac{L}{x}\pi(x^{1/3})+\frac1x\sum_{p\le x^{1/3}}p
=O\Bigl(\frac{x^{1/3}\log x}{x^{1/2}\cdot1}\Bigr)
+O\Bigl(\frac{x^{2/3}}{x\log x}\Bigr)=o(1).
\]
(Indeed $\pi(x^{1/3})=O(x^{1/3}/\log x)$, so the first term is
$O(Cx^{-1/6}\log x)=o(1)$.)

For $x^{1/3}<p\le x^{1/2}$, the integers $m$ counted by $c_p(x)$ are
primes (any composite $m\ge p^2>x^{2/3}$ would exceed
$(x+L)/p\le x^{2/3}+O(x^{1/6}(\log x)^2)$ for large $x$). For \(p>x^{1/3}\), the only composite \(m\) with \(p(m)\ge p\) are those with at least two prime factors \(\ge p\). Such \(m\) satisfy \(m\ge p^2\), and the condition \(m\le (x+L)/p\) implies \(p^3\le x+L\). The contribution of these composites to \(\mu_C(x)\) is \(O\big(\frac{L}{x}\sum_{p\in(x^{1/3},(x+L)^{1/3}]} \frac1p\big)=O(x^{-1/6}(\log x)^2)=o(1)\), hence they can be absorbed into the \(o(1)\) error term. Thus we may write \(c_p(x)=\pi((x+L)/p)-\pi(x/p)+O(1)\) uniformly in the range \(x^{1/3}<p\le x^{1/2}\);
%so $c_p(x)=\pi((x+L)/p)-\pi(x/p)$ up to $O(1)$;
and, by the prime number theorem, its mean value over $x$ is
\begin{equation}\label{eq:lambda}
\lambda_p(x):=\frac{L}{p\log(x/p)}\bigl(1+o(1)\bigr),
\end{equation}
uniformly for $x^{1/3}<p\le x^{1/2}$.

\begin{proof}[Proof of Theorem \ref{thm:mean}]
We first record the interchange of summation. For each composite integer \(n\), the number of \(x\le X\) satisfying
\[
x\le n\le x+L_C(x)
\]
is \(L_C(n)(1+o(1))\) when \(n\le X\), while the boundary contributions \(n<X^{1/2}\) and \(n>X\) are \(o(X)\) (this is the content of Remark 4.1). Hence
\[
\sum_{x\le X}\mu_C(x)
=
(1+o(1))\sum_{\substack{n\le X\\ n\ \mathrm{composite}}}
\frac{p(n)}{n}L_C(n)
+o(X).
\]
Thus it suffices to evaluate the weighted sum on the right-hand side.

Recall that $
L_C(t)=C t^{1/2}(\log t)^2 $ and $S(t)$  introduced in Theorem \ref{thm:c8}. 
Using Abel summation (Stieltjes integration by parts),
\[
\sum_{\substack{n\le X\\ n\ \mathrm{composite}}}
\frac{p(n)}{n}L_C(n)
=
\int_{1^-}^{X}L_C(t)\,dS(t)
=
L_C(X)S(X)-\int_{1}^{X}S(t)L_C'(t)\,dt.
\]

Next, the two terms are estimated separately.

First, by Theorem \ref{thm:c8},
\begin{align*}
&L_C(X)S(X)\\
=&
C X^{1/2}(\log X)^2
\cdot
\frac{8X^{1/2}}{(\log X)^2}
\left(1+O\left(\frac1{\log X}\right)\right)
+O\left(X^{1/2}(\log X)^2\frac{X^{1/3}}{\log X}\right)\\
=&
8CX+O_C\left(\frac{X}{\log X}\right).
\end{align*}
Next, a direct differentiation gives
\[
L_C'(t)
=
C t^{-1/2}(\log t)
\left(\frac12\log t+2\right)
=
\frac{C}{2}t^{-1/2}(\log t)^2
\left(1+O\left(\frac1{\log t}\right)\right).
\]
So, using Theorem \ref{thm:c8} again, one has
\begin{align*}
&S(t)L_C'(t)\\
=&
\left[
\frac{8t^{1/2}}{(\log t)^2}
\left(1+O\left(\frac1{\log t}\right)\right)
+O\left(\frac{t^{1/3}}{\log t}\right)
\right]
\cdot
\frac{C}{2}t^{-1/2}(\log t)^2
\left(1+O\left(\frac1{\log t}\right)\right)\\
=&
4C+O_C\left(\frac1{\log t}\right)
+O\left(t^{-1/6}\log t\right).
\end{align*}
Integrating over \(t\in[1,X]\),
\[
\int_{1}^{X}S(t)L_C'(t)\,dt
=
4CX
+O_C\left(\frac{X}{\log X}\right)
+O\left(X^{5/6}\log X\right).
\]
Combining the two estimates,
\[
\sum_{\substack{n\le X\\ n\ \mathrm{composite}}}
\frac{p(n)}{n}L_C(n)
=
8CX-4CX
+O_C\left(\frac{X}{\log X}\right)
=
4CX+O_C\left(\frac{X}{\log X}\right).
\]
Substituting this back into the interchange formula yields
\[
\sum_{x\le X}\mu_C(x)
=
(1+o(1))\left(4CX+O_C\left(\frac{X}{\log X}\right)\right)
+o(X)
=
4CX+O_C\left(\frac{X}{\log X}\right).
\]
Dividing by \(X\), one obtains
\[
\frac1X\sum_{x\le X}\mu_C(x)
=
4C+O_C\left(\frac1{\log X}\right).
\]
This completes the proof of Theorem \ref{thm:mean}.
\end{proof}

\begin{remark}\label{rem:heuristic_mean}
Interchanging summation,
\[
\sum_{x\le X}\mu_C(x)
=\sum_{n\le X+L(X)}\frac{p(n)}{n}\,
\#\bigl\{x\le X:\ x\le n\le x+L_C(x)\bigr\}.
\]
For $n\in[X^{1/2},X]$ the count of admissible $x$ is
$L_C(n)(1+o(1))$ (the function $L_C$ varies by
$O(L_C(n)^2/n)=o(L_C(n))$ on intervals of length $L_C(n)$).
and the
boundary contributions ($n<X^{1/2}$, $n>X$) are
$O(\sum_{n<X^{1/2}}p(n))=O(X^{3/4})$ and
$O(L(X)\cdot S(X+L))=o(X)$ respectively.
For \(n<X^{1/2}\), the contribution is \(O(L_C(X^{1/2})\sum_{n<X^{1/2}}p(n)/n)=O(X^{3/4}\log X)=o(X)\). For \(n>X\), one has at most \(L_C(X)\) values of \(x\), so the contribution is \(\ll L_C(X)\sum_{n=X}^{X+L}p(n)/n \ll \sqrt{X}(\log X)^4=o(X)\).
Hence
\[
\sum_{x\le X}\mu_C(x)
=\bigl(1+o(1)\bigr)\sum_{n\le X}\frac{p(n)}{n}\,L_C(n)
=C\bigl(1+o(1)\bigr)\sum_{n\le X}\frac{p(n)(\log n)^2}{n^{1/2}}.
\]
\end{remark}

\begin{proof}[Proof of Theorem \ref{thm:var}]
By \eqref{eq:levels} and the pointwise bound for small levels,
\[
\mu_C(x)=\Sigma_1(x)+\Sigma_2(x)+o(1),
\]
where
\[
\Sigma_1(x):=\sum_{x^{1/3}<p\le x^{1/2}}\frac px\,c_p(x),
\qquad
\Sigma_2(x):=\sum_{p\le x^{1/3}}\frac px\,c_p(x)=o(1)
\]
pointwise and uniformly in \(x\).  Let
\[
\bar\Sigma_1:=\frac1X\sum_{x\le X}\Sigma_1(x).
\]
Then
\[
\mu_C(x)-4C
=
\bigl(\Sigma_1(x)-\bar\Sigma_1\bigr)
+
\bigl(\bar\Sigma_1-4C\bigr)
+
\Sigma_2(x)
+
o(1).
\]
Using the elementary inequality \((a+b+c)^2\le 3(a^2+b^2+c^2)\), one obtains
\[
\frac1X\sum_{x\le X}\bigl(\mu_C(x)-4C\bigr)^2
\le
3\left(
\frac1X\sum_{x\le X}\bigl(\Sigma_1(x)-\bar\Sigma_1\bigr)^2
+
\bigl(\bar\Sigma_1-4C\bigr)^2
+
\frac1X\sum_{x\le X}\Sigma_2(x)^2
\right)
+
o(1).
\]

The three terms are estimated separately.

\medskip
\noindent
{\bf \emph{1. Mean bias.}}
By Theorem \ref{thm:mean},
\[
\bigl|\bar\Sigma_1-4C\bigr|
\ll_C
\frac1{\log X},
\]
hence
\[
\bigl(\bar\Sigma_1-4C\bigr)^2
=
O_C\!\left(\frac1{(\log X)^2}\right).
\]

\medskip
\noindent
{\bf\emph{2. Small levels.}}
Since \(\Sigma_2(x)=o(1)\) pointwise and
\[
|\Sigma_2(x)|
\le
\frac{L_C(x)}x\pi(x^{1/3})
+
\frac1x\sum_{p\le x^{1/3}}p
\ll_C
x^{-1/6}\log x+x^{-1/3}(\log x)^{-1}
=
o(1),
\]
one has
\[
\frac1X\sum_{x\le X}\Sigma_2(x)^2=o(1).
\]

\medskip
\noindent
{\bf\emph{3. The main level term.}}
The bound of the following item is studied
\[
D:=\frac1X\sum_{x\le X}\bigl(\Sigma_1(x)-\bar\Sigma_1\bigr)^2.
\]
Write
\[
\Sigma_1(x)-\bar\Sigma_1
=
\sum_{x^{1/3}<p\le x^{1/2}}
\frac px\bigl(c_p(x)-\bar c_p\bigr)
+
O_C\!\left(\frac1{\log X}\right),
\]
where
\[
\bar c_p:=\frac1X\sum_{x\le X}c_p(x).
\]
The error term is absorbed into the mean bias.  Hence, up to \(O_C((\log X)^{-2})\),
\[
D
\le
\sum_{\substack{p,p'\in(X^{1/3},X^{1/2}]}}
\frac{pp'}{X^2}
\cdot
\frac1X\sum_{x\le X}
\bigl(c_p(x)-\bar c_p\bigr)
\bigl(c_{p'}(x)-\bar c_{p'}\bigr)
+
o(1).
\]
We split the sum into the diagonal part \(p=p'\) and the cross terms \(p\ne p'\).

\smallskip
\noindent
{\bf\emph{3a. Diagonal part.}}
For \(p\in(X^{1/3},X^{1/2}]\),
\[
\frac1X\sum_{x\le X}c_p(x)^2
=
\frac1X\sum_{q,q'\le X/p}
\bigl(L-p|q-q'|\bigr)_+
\le
\frac1X\pi\!\left(\frac Xp\right)
\left(\frac{2L}{p}+1\right)L
\ll
\frac{L^2}{p^2\log(X/p)}.
\]
Also
\[
\left(\frac1X\sum_{x\le X}c_p(x)\right)^2
\le
\lambda_p^2
\ll
\frac{L^2}{p^2(\log(X/p))^2}.
\]
Therefore
\[
\frac1X\sum_{x\le X}\bigl(c_p(x)-\bar c_p\bigr)^2
\ll
\frac{L^2}{p^2\log(X/p)}.
\]
Summing over \(p\),
\[
\sum_{X^{1/3}<p\le X^{1/2}}
\left(\frac pX\right)^2
\frac{L^2}{p^2\log(X/p)}
=
\frac{L^2}{X^2}
\sum_{X^{1/3}<p\le X^{1/2}}
\frac1{\log(X/p)}.
\]
By Lemma \ref{lem:part2},
\[
\sum_{X^{1/3}<p\le X^{1/2}}\frac1{\log(X/p)}
\ll
\frac{X^{1/2}}{(\log X)^2}.
\]
Since \(L=CX^{1/2}(\log X)^2\),
\[
\frac{L^2}{X^2}\cdot
\frac{X^{1/2}}{(\log X)^2}
=
\frac{C^2X(\log X)^4}{X^2}\cdot
\frac{X^{1/2}}{(\log X)^2}
=
\frac{C^2(\log X)^2}{X^{1/2}}
=
o(1).
\]
Thus the diagonal contribution is \(o(1)\) (in fact \(o((\log X)^{-2})\)).

\smallskip
\noindent
{\bf\emph{3b. Cross terms.}}
For \(p\ne p'\), we use the crude covariance bound.  For fixed \(q\) counted by
\(c_p(x)\) and fixed \(x\), the primes \(q'\) counted by \(c_{p'}(x)\) satisfy
\[
|pq-p'q'|\le L_C(x)\le L_C(X).
\]
Hence \(q'\) lies in an interval of length \(2L_C(X)/p'\).  The number of
primes in such an interval is
\[
\ll
\frac{L_C(X)}{p'\log(X/p')}+1.
\]
Therefore
\[
\mathbb{E}_X[c_pc_{p'}]
\ll
\lambda_p
\left(
\frac{L_C(X)}{p'\log(X/p')}+1
\right),
\qquad
\lambda_p\asymp\frac{L_C(X)}{p\log(X/p)}.
\]
Since \(\mathbb{E}_X[c_p]\mathbb{E}_X[c_{p'}]\ge0\),
\[
\operatorname{Cov}(c_p,c_{p'})
\ll
\lambda_p
\left(
\frac{L_C(X)}{p'\log(X/p')}+1
\right).
\]
Thus the cross contribution is bounded by \(S_A+S_B\), where
\[
S_A
:=
\sum_{p,p'}
\frac{pp'}{X^2}\,
\lambda_p\,
\frac{L_C(X)}{p'\log(X/p')},
\]
and
\[
S_B
:=
\sum_{p,p'}
\frac{pp'}{X^2}\,
\lambda_p.
\]
(The sums run over \(X^{1/3}<p,p'\le X^{1/2}\); the restriction \(p\ne p'\) is
irrelevant for an upper bound.)

For \(S_A\), we use
\[
\lambda_p\asymp\frac{L_C(X)}{p\log(X/p)}.
\]
Then
\[
S_A
\ll
\frac{L_C(X)^2}{X^2}
\sum_{p,p'}
\frac1{\log(X/p)\log(X/p')}
=
\frac{L_C(X)^2}{X^2}
\left(
\sum_{X^{1/3}<p\le X^{1/2}}
\frac1{\log(X/p)}
\right)^2.
\]
By Lemma \ref{lem:part2},
\[
\sum_{X^{1/3}<p\le X^{1/2}}\frac1{\log(X/p)}
\ll
\frac{X^{1/2}}{(\log X)^2}.
\]
Hence
\[
S_A
\ll
\frac{L_C(X)^2}{X^2}
\cdot
\frac{X}{(\log X)^4}
=
\frac{C^2X(\log X)^4}{X^2}
\cdot
\frac{X}{(\log X)^4}
=
C^2.
\]
Therefore
\[
S_A=O_C(1).
\]

For \(S_B\), one has
\[
S_B
=
\frac{L_C(X)}{X^2}
\left(
\sum_{X^{1/3}<p\le X^{1/2}}
\frac{p}{p\log(X/p)}
\right)
\left(
\sum_{X^{1/3}<p'\le X^{1/2}}p'
\right).
\]
The first factor is
\[
\sum_{X^{1/3}<p\le X^{1/2}}\frac1{\log(X/p)}
\ll
\frac{X^{1/2}}{(\log X)^2}.
\]
The second factor satisfies
\[
\sum_{X^{1/3}<p'\le X^{1/2}}p'
\ll
\frac{X}{\log X}.
\]
Therefore
\[
S_B
\ll
\frac{L_C(X)}{X^2}
\cdot
\frac{X^{1/2}}{(\log X)^2}
\cdot
\frac{X}{\log X}
=
\frac{CX^{1/2}(\log X)^2}{X^2}
\cdot
\frac{X^{3/2}}{(\log X)^3}
=
\frac{C}{\log X}.
\]
Thus
\[
S_B=O_C\!\left(\frac1{\log X}\right).
\]
Combining,
\[
D
\le
S_A+S_B+o(1)
=
O_C(1)+O_C\!\left(\frac1{\log X}\right)+o(1)
=
O_C(1).
\]

Collecting the estimates, one has
\[
\frac1X\sum_{x\le X}\bigl(\mu_C(x)-4C\bigr)^2
\le
3\left(
O_C(1)
+
O_C\!\left(\frac1{(\log X)^2}\right)
+
o(1)
\right)
+
o(1)
=
O_C(1).
\]
This proves the theorem.
\end{proof}

\section{Uniformity for all $x$}\label{sec:uniform}

\begin{proof}[Proof of Theorem \ref{thm:uniform}]
Assume Hypothesis \ref{hyp:SI}. Fix \(C>0\) and put
\[
L=L_C(x):=C x^{1/2}(\log x)^2,\qquad I:=[x,x+L].
\]

We use the level decomposition from Section \ref{sec:windows}:
\begin{equation}
\mu_C(x)=\sum_{p\le (x+L)^{1/2}}\frac{p}{x}\,c_p(x)+o(1),
\label{eq:level}
\end{equation}
where
\[
c_p(x):=\#\bigl\{m: p(m)\ge p,\; x/p\le m\le (x+L)/p\bigr\}.
\]
The term \(o(1)\) is uniform in \(x\) and comes from the prime \(m=1\) and endpoint effects; it will be absorbed into the final error.

The sum over primes \(p\) is divided into three ranges.

{\bf{Case 1.}} Small primes: \(p\le x^{1/3}\)

By \eqref{eq:levels} and the pointwise
bound of Section \ref{sec:windows},
(which is uniform in \(x\)), one has
\[
\sum_{p\le x^{1/3}}\frac{p}{x}\,c_p(x)
\le
\frac{L}{x}\pi(x^{1/3})+\frac1x\sum_{p\le x^{1/3}}p
\ll
\frac{Cx^{1/2}(\log x)^2}{x}\cdot \frac{x^{1/3}}{\log x}
+\frac{1}{x}\cdot \frac{x^{2/3}}{\log x}
\ll
C x^{-1/6}\log x + x^{-1/3}/\log x,
\]
which tends to \(0\) uniformly in \(x\) for fixed \(C\). Hence
\[
\sum_{p\le x^{1/3}}\frac{p}{x}c_p(x)=o(1)
\]
uniformly, and this contribution can be absorbed into the final \(O_C(1)\) error.

{{\bf Case 2.}} Medium primes: \(x^{1/3}<p\le x^{1/2}\)

For such \(p\), any composite \(m\) with \(p(m)\ge p\) has at least two prime factors \(\ge p\), so \(m\ge p^2\). The condition \(m\le (x+L)/p\) then implies \(p^3\le x+L\), i.e. \(p\le (x+L)^{1/3}\). The contribution of these composite \(m\)'s is
\begin{align*}
&\sum_{x^{1/3}<p\le (x+L)^{1/3}}\frac{p}{x}\,
\#\{m\ge p^2: p(m)\ge p,\; m\le (x+L)/p\}\\
\ll&
\frac{L}{x}\sum_{x^{1/3}<p\le (x+L)^{1/3}}\frac{1}{p}
\ll
\frac{L}{x}\cdot \frac{x^{1/6}}{\log x}
=
O\!\left(C x^{-1/3}(\log x)^2\right)=o(1).
\end{align*}
Thus, up to a uniform \(o(1)\) error, one may write
\begin{equation}
c_p(x)=\pi\!\left(\frac{x+L}{p}\right)-\pi\!\left(\frac{x}{p}\right)+O(1).
\label{eq:cp}
\end{equation}
The \(O(1)\) term, when multiplied by \(p/x\) and summed over \(p\le x^{1/2}\), gives
\[
O\!\left(\frac{1}{x}\sum_{p\le x^{1/2}}p\right)
=
O\!\left(\frac{1}{\log x}\right),
\]
which is acceptable.

Set
\[
y_{p}:=\frac{x}{p},\qquad h_{p}:=\frac{L}{p}.
\]

Now, the claim \(h_p\ge C_0(\log y_p)^2\) is shown.

Fix \(p\in(x^{1/3},x^{1/2}]\). Write
\[
p=x^{1/2}e^{-s},
\qquad
0\le s\le \tfrac16\log x,
\]
where the upper bound on \(s\) comes from \(p\ge x^{1/3}\). Then
\[
y_p:=\frac xp=x^{1/2}e^{s},
\qquad
\log y_p=\frac12\log x+s,
\]
and
\[
h_p:=\frac Lp
=
\frac{Cx^{1/2}(\log x)^2}{x^{1/2}e^{-s}}
=
Ce^{s}(\log x)^2.
\]
So,
\[
\frac{h_p}{(\log y_p)^2}
=
\frac{Ce^{s}(\log x)^2}{\bigl(\tfrac12\log x+s\bigr)^2}
=
4C\cdot
\frac{e^{s}}{\bigl(1+\tfrac{2s}{\log x}\bigr)^2}.
\]
We claim that
\begin{equation}\label{equation-2026-10-21-55}
e^{s}\ge \Bigl(1+\frac{2s}{\log x}\Bigr)^2
\qquad
\Bigl(0\le s\le \tfrac16\log x\Bigr),
\end{equation}
which will give
\[
\frac{h_p}{(\log y_p)^2}\ge 4C,
\]
that is,
\[
h_p\ge 4C(\log y_p)^2.
\]

To prove \eqref{equation-2026-10-21-55}, put \(v:=s/\log x\in[0,\tfrac16]\) and define
\[
f(v):=v\log x-2\log(1+2v).
\]
Then \(f(0)=0\), and for \(\log x\ge 4\),
\[
f'(v)
=
\log x-\frac{4}{1+2v}
\ge
\log x-4
\ge
0
\qquad
\Bigl(0\le v\le\tfrac16\Bigr).
\]
Hence \(f(v)\ge 0\) on \([0,\tfrac16]\), i.e.
\[
v\log x\ge 2\log(1+2v).
\]
Exponentiating gives
\[
e^{v\log x}=e^{s}\ge (1+2v)^2=\Bigl(1+\frac{2s}{\log x}\Bigr)^2,
\]
which is \eqref{equation-2026-10-21-55}.

Consequently, for every \(p\in(x^{1/3},x^{1/2}]\),
\[
h_p\ge 4C(\log y_p)^2.
\]
Taking \(C_0:=4C\) (or any fixed \(C_0\le 4C\)), the hypothesis
\[
C_0(\log y_p)^2\le h_p\le y_p
\]
of Hypothesis 1.5 is satisfied for all sufficiently large \(x\), uniformly in
\(p\in(x^{1/3},x^{1/2}]\). The upper bound \(h_p\le y_p\) follows from
\[
\frac{h_p}{y_p}
=
\frac{Cx^{1/2}(\log x)^2}{p}\cdot\frac{p}{x}
=
\frac{C(\log x)^2}{x^{1/2}}
\to 0.
\]

Applying Hypothesis \ref{hyp:SI} at level $p$, there is a ratio
$\rho_p\in[c_1,c_2]$ such that
\[
\pi(y_p+h_p)-\pi(y_p)=\rho_p\,\frac{h_p}{\log y_p},
\qquad c_1\le \rho_p\le c_2,
\]
where $c_1=c_1(C)$, $c_2=c_2(C)$ depend only on $C$ (through
$C_0=C$).

Substituting \(h_p=L/p\) and \(\log y_p=\log(x/p)\), we get
\begin{equation}
c_p(x)=\rho_p\,\frac{L}{p\log(x/p)}+O(1),
\qquad c_1\le\rho_p\le c_2,
\label{eq:cp_approx}
\end{equation}
with the $O(1)$ term uniform in $p$.

Summing \eqref{eq:cp_approx} over \(p\in(x^{1/3},x^{1/2}]\),
\[
\mu_C(x)=\frac{L}{x}\sum_{x^{1/3}<p\le x^{1/2}}\frac{\rho_p}{\log(x/p)}
+O\!\left(\frac1{\log x}\right)+o(1).
\]
Writing $\rho_p=1+(\rho_p-1)$, the first term equals
\begin{equation}
\frac{L}{x}\sum_{x^{1/3}<p\le x^{1/2}}\frac{1}{\log(x/p)}
=
4C+O_C\!\left(\frac1{\log x}\right),
\label{eq:main}
\end{equation}
by the same computation as in Lemma \ref{lem:part2} (the integral
evaluation), while the fluctuation term is bounded by
\[
\max\{|c_1-1|,|c_2-1|\}\cdot
\frac{L}{x}\sum_{x^{1/3}<p\le x^{1/2}}\frac{1}{\log(x/p)}
\le
4C\cdot\max\{|c_1-1|,|c_2-1|\}\cdot(1+o(1))
=O_C(1).
\]
The $O(1)$ error from \eqref{eq:cp} contributes $O(1/\log x)$ as
noted above. Hence the medium-prime contribution equals
\[
4C+O_C(1).
\]

{{\bf Case 3.}} Large primes: \(p>x^{1/2}\)

For \(p>\sqrt{x}\), any \(m\ge2\) with \(p(m)\ge p\) satisfies \(m\ge p\), so \(n=pm\ge p^2>x\). The condition \(n\le x+L\) then implies \(p\le \sqrt{x+L}\). Thus only primes \(p\in(\sqrt{x},\sqrt{x+L}]\) can contribute, and for each such \(p\), the possible \(m\) are \(m=1\) (giving \(n=p\), but \(n\) must be composite; however, this term is already excluded in \(\mu_C(x)\)) or \(m\ge p\) with \(pm\le x+L\). The number of such \(m\) is at most \(O((x+L)/p^2)\le O(1)\) (since \(p^2\ge x\)). Therefore \(c_p(x)=O(1)\) for these \(p\). Hence the contribution is
\begin{align*}
&\sum_{\sqrt{x}<p\le \sqrt{x+L}}\frac{p}{x}\,O(1)
\ll
\frac{1}{x}\sum_{\sqrt{x}<p\le \sqrt{x+L}}p
\ll
\frac{1}{x}\cdot \sqrt{x}\cdot \#\{p\}
\ll
\frac{1}{x}\cdot \sqrt{x}\cdot \frac{L/\sqrt{x}}{\log x}\\
=&
\frac{L}{x\log x}
=
\frac{Cx^{1/2}(\log x)^2}{x\log x}
=
\frac{C\log x}{x^{1/2}}=o(1).
\end{align*}
Thus the large-prime contribution is uniformly \(o(1)\).

Combining the three ranges, we obtain
\[
\mu_C(x)
=
\frac{L}{x}\sum_{x^{1/3}<p\le x^{1/2}}\frac{\rho_p}{\log(x/p)}
+
O\!\left(\frac1{\log x}\right)
+
o(1).
\]
Using \eqref{eq:main} and the fluctuation bound above, and absorbing
the \(o(1)\) into the \(O_C(1)\), we get
\[
\mu_C(x)=4C+O_C(1)
\]
uniformly for all sufficiently large \(x\). This proves Theorem \ref{thm:uniform}.
\end{proof}

\end{document}